\documentclass{article}
\usepackage{xcolor}
\usepackage{graphicx}
\usepackage{hyperref}
\usepackage{amsfonts,amsmath,amsthm}
\usepackage{natbib} 
\newtheorem{theorem}{Theorem}

\newtheorem{lemma}{Lemma}

\newtheorem{remark}{Remark}

\newcommand{\indicator}{\mathbf{1}}
\newcommand{\F}{\mathcal{F}}

\newcommand{\R}{\mathbb{R}}
\newcommand{\M}{\mathcal{M}}
\renewcommand\P{\operatorname{\mathbf{P}}}
\newcommand\E{\operatorname{\mathbf{E}}}

\begin{document}
 \title{Two-sided optimal stopping for L\'evy processes}
 \author{Ernesto  Mordecki and Facundo Oli\'u Eguren}
 \maketitle
 \begin{abstract}
Infinite horizon optimal stopping problems for a L\'evy processes with a two-sided reward function are considered.
A two-sided verification theorem is presented in terms of the overall supremum and the overall infimum of the process.
A result to compute the angle of the value function at the optimal thresholds of the stopping region is given.
To illustrate the results,  the optimal stopping problem of a compound Poisson process with two-sided exponential jumps 
and a two-sided payoff function is solved.
In this example, the smooth-pasting condition does not hold.
\end{abstract}
\section{Introduction}
The presence of the maximum in the solution of optimal stopping problems for L\'evy processes
is nowadays understood, as it was summarized in the monograph by \cite{Kyprianou}.
After the work of \cite{Surya} (see also \cite{MordeckiMishura}) it became clear that
one sided problems with arbitrary payoff functions could be solved with the help of an \emph{averaging function}.
More recently, some formulas appeared for two sided problems, where the infimum and the supremum took part in the solution.
For instance, in \cite{MordeckiSalminen}, with the help of representation techniques, 
a formula for the value function through the sum two averaging functions
(one for the maximum and another for the minimum) 
was obtained. Afterwards, a similar formula was obtained through verification techniques
by \cite{ChristensenEtAl}. This second result involved the supremum of the two averaging functions.


The purpose of the present work is then to obtain a verification theorem for the optimal stopping problem of a 
L\'evy process in the two sided case, through  the sum of two averaging functions,
and to provide a simple example that seems not possible to be solved with the existing techniques. 


References on optimal stopping problems for L\'evy processes with two sided solutions, 
to our knowledge, are few.
Perpetual American strangle options are considered by \cite{ChangSheu},
through the solution of a free boundary integro-diffe\-ren\-tial problem with moving boundaries.
In this case, the smooth pasting condition is a key ingredient, and a non vanishing gaussian part in the
L\'evy process is used to ensure this condition.
These results follow previously obtained unpublished  ones by \cite{Boyarchenko}.
Similar type of techniques were applied in \cite{BuonaguidiMuliere} to solve a statistical problem:
the Bayesian sequential testing of two simple hypothesis.
More recently, \cite{Palmowski} found disconnected continuation regions in American put options with negative
discount rates in L\'evy  models.

The content of the rest of the paper is as follows. 
In section \ref{section:2} we formulate the verification result for two-sided optimal stopping with the corresponding proof. 
In section \ref{section:angle} a result to compute the angle (i.e. the difference of the right and left derivatives) of the value function at a critical threshold is obtained.
Section \ref{section:example} contains an example:
the optimal stopping problem of a compound Poission process with two-sided exponential jumps,
(no gaussian component) and payoff function $g(x)=|x|$.
%


\section{A verification result for optimal stopping}\label{section:2}

Let $X=\{X_t\colon t\geq 0\}$ be a L\'evy process defined on a
stochastic basis 
$(\Omega, {\cal F}, {\bf F}=({\cal F}_t)_{t\geq 0}, \P_x)$ 
departing from $X_0=x$. The corresponding expectation is denoted by $\E_x$. 
The L\'evy-Khintchine formula characterizes the law of the process,
stating, for $z\in i\R$, that 
$
\E_0 e^{zX_t}=e^{t\psi(z)}
$
with
\begin{equation*}
\psi(z)=az+{\sigma^2\over 2}z^2+\int_{\R}\left(e^{zy}-1-zh(y)\right)\Pi(dy),
\end{equation*}
where $a\in\R$, $\sigma\geq 0$ and $\Pi(dy)$ is a non-negative measure (the \emph{jump measure}) 
that satisfies $\int_{\R}(1\wedge y^2)\Pi(dy)<\infty$.
Here $h(y)=y\indicator_{\{|y|<1\}}$ is a truncation function.
For general references on L\'evy processes see \cite{Bertoin} or \cite{Kyprianou}.
The set of stopping times is the set of random variables
$$
\M=\{\tau\colon\Omega\to[0,\infty] \text{ such that } \{\tau\leq t\}\in\mathcal{F}_t \text{ for all $t\geq 0$}\}.
$$
Observe that we allow the possibility $\tau=\infty$ as for several optimal stopping problems, the optimal stopping time
is within this class.
A key r\^ole in the solution of two-sided problems is played by the overall supremum and infimum of the process,
defined respectively by
\begin{equation*}
M=\sup\{X_t\colon 0\leq t\leq e_r\},\quad
I=\inf\{X_t\colon 0\leq t\leq e_r\},
\end{equation*}
where $e_r$ is an exponential random variable of parameter $r>0$, independent of $X$.
Observe that as $r>0$ both random variables $M$ and $I$ are proper.

Given a non-negative continuous payoff function $g(x)$,
a L\'evy process $X$, 
and a discount factor $r>0$, the optimal stopping problem (OSP) consists
in finding the value function ${V}(x)$ and the optimal stopping rule ${\tau^*}$ such that
\begin{equation}\label{eq:osp}
{V}(x)=\sup_{\tau\in\M}\E( e^{-r\tau}g(X_{\tau}))=\E (e^{-r{\tau^*}}g(X_{\tau^*})).
\end{equation}
We assume that the payoff received in the set
$\{\omega\colon\tau(\omega)=\infty\}$ is zero, in fact, we identify
$$
e^{-r\tau}g(X_{\tau})=e^{-r\tau}g(X_{\tau})\indicator_{\{\tau<\infty\}}.
$$
In the present paper we are interested in problems with two-sided solutions, 
i.e. such that the optimal stopping rule is of the form
\begin{equation}\label{eq:tau0}
{\tau^*}=\inf\{t\geq 0\colon X_t\notin(-x_1, x_2)\},
\end{equation}
for some critical thresholds $-x_1<0<x_2$.
As the process is space-invariant, the thresholds are chosen negative and positive for simple convenience of notation.  
Observe that, having into account the asymptotic behavior of a  L\'evy process 
(see Thm. VI.12 in \cite{Bertoin}) the stopping time in \eqref{eq:tau0} satisfies
$\P({\tau^*}<\infty)=1$.

\begin{theorem}\label{theorem:1}
Consider a L\'evy process $X$,
a discount rate $r>0$, and a continuous reward function $g\colon\R\to[0,\infty)$.
Assume that there exist two points $-x_1<0<x_2$ and two continuous monotonous functions: 
$Q_1$ non-increasing with $Q_1(x)=0$ for $ -x_1\leq x$; 
$Q_2$ non-decreasing with $Q_2(x)=0$ for $x\leq x_2$
(named \emph{averaging} functions);
and such that
\begin{equation}\label{eq:equal}
g(x)=\E_x Q_1(I)+\E_x Q_2(M), \text{ for all $x\notin(-x_1,x_2)$.}
\end{equation}
Define the function
\begin{equation}\label{eq:ve}
V(x)=\E_x (Q_1(I)+ Q_2(M)), \text{ for all $x\in\R$.}
\end{equation}
Then, if the condition
\begin{equation}\label{eq:geq}
V(x)\geq g(x),
\end{equation}
holds for all $x\in[-x_1,x_2]$, the OSP \eqref{eq:osp}
has value function $V(x)$ in \eqref{eq:ve}, and \eqref{eq:tau0} is an optimal stopping time for the problem.
\end{theorem}
As usual in optimal stopping, in order to prove Theorem \ref{theorem:1}, we verify two statements:
\begin{align}
V(x)&\geq\E( e^{-r\tau}g(X_{\tau})),
\quad\forall\tau\in\M,\label{eq:inequality}\\
&\notag\\
V(x)&=\E( e^{-r{\tau^*}}g(X_{\tau^*})).\label{eq:equality}
\end{align}
The proof of these two facts are stated in two corresponding lemmas.
Their proofs follow, with the necessary modifications,
from the  respective proofs in \cite{MordeckiMishura}. 
\begin{lemma}\label{lemma:two}
For a L\'evy process and $r>0$, consider two non-negative continuous functions: 
$f(x)$ non-decreasing; 
$g(x)$ non-increasing.
Then:
\begin{itemize}
\item[\rm (a)]
The function
\[
h(x)=\E (f({M})+g(I))\quad (x\in\R)
\]
is $r$-excessive, and, in consequence,
\item[\rm (b)]
the process
$\{e^{-{r}t}h(X_t)\colon t\geq 0\}$ is a supermartingale.
\end{itemize}
\end{lemma}
\begin{proof}
The fact that (b) follows from (a) is standard, see for
example \cite{Shiryaev}.
To prove (a) we know, from Lemma 2.2  in \cite{ChristensenEtAl}, that the following two functions
are excessive:
\begin{align*}
u(x):=\E_x\sup_{0\leq t\leq e_r}f(X_t)&=\E_x f\left(\sup_{0\leq t\leq e_r}X_t\right)=\E_xf(M),\\
v(x):=\E_x\sup_{0\leq t\leq e_r}g(X_t)&=\E_xg\left(\inf_{0\leq t\leq e_r}X_t\right)=\E_xg(I).\\
\end{align*}
The proof concludes as $h=u+v$, and excessivity is preserved by summation.
\end{proof}
\begin{remark}
When comparing the representations of payoffs for two-sided problems in 
Theorem 2.7 in \cite{ChristensenEtAl} and Proposition 4.4. in \cite{MordeckiSalminen},
it should be noticed that in the first case, the excessive function is $\E_x(f(M)\vee g(I))$, and in the second,
$\E_x(f(M)+g(I))$. These two constructions give different representations of the payoff of the problem.
In the present work we use the second one.
\end{remark}
\begin{lemma}\label{lemma:1}
Consider a L\'evy process $X$,
a discount rate $r>0$, and functions $Q_1,Q_2$ and $g$ such that \eqref{eq:equal} and \eqref{eq:geq} hold,
for $V$ defined by \eqref{eq:ve}. Then \eqref{eq:equality} holds.
\end{lemma}
\begin{proof}
Denote $S=\R\setminus(-x_1,x_2)$. 
As $X_{\tau^*}\in S$ and $V=g$ on $S$, we have
\begin{equation}\label{eq:vege}
\E_x(e^{-r{\tau^*}}g(X_{\tau^*}))=\E_x(e^{-r{\tau^*}}V(X_{\tau^*})).
\end{equation}
On the other side, 
$$
V(x)=\E_x (Q_1(I)+Q_2(M))=\E_x\left(\sup_{0\leq t\leq e_r}Q_1(X_t)+\sup_{0\leq t\leq e_r}Q_2(X_t)\right).
$$
Observe that if either $e_r<{\tau^*}$ or $t<{\tau^*}$, we have $M<x_2$ and $I>-x_1$, and then 
\begin{equation*}
\sup_{0\leq t\leq e_r}Q_1(X_t)+\sup_{0\leq t\leq e_r}Q_2(X_t)=0,
\end{equation*}
because $Q_1(x)=Q_2(x)=0$ in $[-x_1,x_2]$.
So, denoting
\begin{align}
V_1(x)&:=\E_x\left(\sup_{{\tau^*}\leq t\leq e_r}Q_1(X_t)\right)=\E_xQ_1(I),\label{eq:inf}\\
V_2(x)&:=\E_x\left(\sup_{{\tau^*}\leq t\leq e_r}Q_2(X_t)\right)=\E_xQ_2(M).\label{eq:sup}
\end{align}
we have
$$
V(x)=V_1(x)+V_2(x).
$$
Consider now $\tilde{X}=\{\tilde{X}_s=X_{{\tau^*}+s}-X_{\tau^*}\colon s\geq 0\}$ 
that, by the strong Markov property,
is independent of $\F_{\tau^*}$ and has the same distribution as $X$, 
and denote by $\tilde{\E}_x$ the expectation w.r.t. $\tilde{X}$.
Based on these considerations, we have
\begin{align}
V_1(x)&=\E_x\left(\sup_{{\tau^*}\leq t\leq e_r}Q_1(X_t)\right)
=\E_x\left(\int_{\tau^*}^{\infty}\sup_{{\tau^*}\leq t\leq u}Q_1(X_t)re^{-ru}du\right)\label{eq:u}\\
&=\E_x\left(e^{-r{\tau^*}}\int_0^{\infty}\sup_{{\tau^*}\leq t\leq {\tau^*}+v}Q_1(X_t)re^{-rv}dv\right)\label{eq:v}\\
&=\E_x\left(e^{-r{\tau^*}}\int_0^{\infty}\sup_{{\tau^*}\leq t\leq {\tau^*}+v}Q_1(X_{\tau^*}+X_t-X_{\tau^*})re^{-rv}dv\right)\label{eq:t}\\
&=\E_x\left(e^{-r{\tau^*}}\int_0^{\infty}\sup_{0\leq s\leq v}Q_1(X_{\tau^*}+X_{{\tau^*}+s}-X_{\tau^*})re^{-rv}dv\right)\label{eq:s}\\
&=\E_x\left(e^{-r{\tau^*}}\int_0^{\infty}\sup_{0\leq s\leq v}Q_1(X_{\tau^*}+\tilde{X}_s)re^{-rv}dv\right)\notag\\
&=\E_x\left(e^{-r{\tau^*}}\tilde{\E}_{X_{\tau^*}}\left[\int_0^{\infty}\sup_{0\leq s\leq v}Q_1(\tilde{X}_s)re^{-rv}dv\right]\right)\notag\\
&=\E_x\left(e^{-r{\tau^*}}\tilde{\E}_{X_{\tau^*}}\left[\sup_{0\leq s\leq e_r}Q_1(\tilde{X}_s)\right]\right)=
\E_x\left(e^{-r{\tau^*}}V_1(X_{\tau^*})\right),\notag
\end{align}
where we change variables according to $v=u-{\tau^*}$ to pass from \eqref{eq:u} to \eqref{eq:v},
and denote $s=t-{\tau^*}$ to pass from \eqref{eq:t} to \eqref{eq:s}. 
The same relation holds with $V_2$ and $Q_2$.
Summing up these two relations, and in view of \eqref{eq:vege},
we conclude the proof of the Lemma.
\end{proof}
\begin{proof}[Proof of Theorem \ref{theorem:1}] The proof now follows, as, according to Lemma \ref{lemma:two}
the non-negative process $\{e^{-rt}V(X_t)\colon t\geq 0\}$  is  a supermartingale, giving 
$$
V(x)\geq\E( e^{-r\tau}V(X_{\tau}))\geq\E( e^{-r\tau}g(X_{\tau})),
$$
by Doob's optional sampling Theorem first and the application of condition \eqref{eq:geq} second.
This establishes \eqref{eq:inequality}.
On its turn, Lemma \ref{lemma:1} gives equality \eqref{eq:equality}, concluding the proof.
\end{proof}
\section{On the smooth pasting condition}\label{section:angle}
Smooth pasting results for general L\'evy processes were obtained for put American perpetual options for L\'evy processes in \cite{AliliKyprianou}, and afterwards generalized to put-type (bounded) payoffs by \cite{Surya}. 
The following result gives some natural necessary conditions for smooth pasting, that
depend on the exponential moments of the process and the behaviour of the averaging function.
They can be applied both to one-sided problems (as the ones considered in \cite{MordeckiMishura})
and to the two sided problems considered in the present paper.
\begin{theorem}\label{theorem:angle}
Consider a L\'evy process $X$,
a discount rate $r>0$ and a continuous reward function $g\colon\R\to[0,\infty)$.
Assume that there exist a point $x_0$ and a continuous non-decreasing averaging function 
$Q$ with $Q(x)=0$ for $x\leq x_0$,
such that
\begin{equation*}
g(x)=\E_x Q(M), \text{ for all $x\geq x_0$.}
\end{equation*}
Assume that $Q\in C^2[x_0,\infty)$,  and satisfies
\begin{equation*}
|Q''(x)|\leq Ae^{\alpha x},\ \forall x\geq x_0,
\end{equation*}
for some $A>0$ and some $\alpha>0$.
Regarding the process, assume that 
\begin{equation}\label{eq:r}
\E e^{\alpha X_1}<e^{r}.
\end{equation}
Then, the candidate to value function of the OSP \eqref{eq:osp}
\begin{equation*}
V(x)=\E_x Q(M), \text{ for all $x\in\R$}
\end{equation*}
satisfies
\begin{equation}\label{eq:smooth}
V'(x_0+)-V'(x_0-)=Q'(x_0+)\P(M=0),
\end{equation}
and
\begin{equation}\label{eq:smooth2}
V'(x+)-V'(x-)=0,\quad\text{for $x>x_0$}.
\end{equation}

\end{theorem}
\begin{proof}
We first prove \eqref{eq:smooth}.
Denoting by $F_M(y)\ (y\geq 0)$ the distribution function of $M$, we have
\begin{align}
V'(&x_0+)-V'(x_0-)\notag\\
&=\lim_{h\downarrow 0}\frac1h
\E[Q(x_0+h+M)+Q(x_0-h-M)-2Q(x_0+M)]\notag\\
&=\lim_{h\downarrow 0}\frac1h\int_{[0,\infty)}\left[Q(x_0+h+y)+Q(x_0-h-y)-2Q(x_0+y)\right]dF_M(y)\notag\\
&=\lim_{h\downarrow 0}\frac1h
\Big\{
[Q(x_0+h)+Q(x_0-h)-2Q(x_0)]\P(M=0)\Big.\label{eq:aa}\\
&\quad+\int_{(0,h)}
\left[Q(x_0+h+y)+Q(x_0-h+y)-2Q(x_0+y)\right]dF_M(y)\label{eq:bb}\\
\Big.&\quad+\int_{(h,\infty)}
\left[Q(x_0+h+y)+Q(x_0-h+y)-2Q(x_0+y)\right]dF_M(y)\Big\}.\label{eq:cc}
\end{align}
To compute the limit in \eqref{eq:aa}, as $Q(x)=0$ for $x\leq x_0$, we have
$$
\lim_{h\downarrow 0}\frac1h
[Q(x_0+h)+Q(x_0-h)-2Q(x_0)]\P(M=0)=Q'(x_0+)\P(M=0).
$$
Concerning \eqref{eq:bb}, we have
\begin{multline*}
\lim_{h\downarrow 0}\frac1h\left|\int_{(0,h)}
\left[Q(x_0+h+y)+Q(x_0-h+y)-2Q(x_0+y)\right]dF_M(y)\right|\\
\leq
4\lim_{h\downarrow 0}\frac{Q(x_0+2h)}h\P(0<M\leq h)=
8Q(x_0)'\lim_{h\downarrow 0}\P(0<M\leq h)=0.
\end{multline*}
To consider the term in \eqref{eq:cc}, denote $x=x_0+y\geq h$,
\begin{multline*}
|Q(x+h)+Q(x-h)-2Q(x)|
=\left|\int_{x}^{x+h}du\int_{u-h}^{u}Q''(v)dv\right|\\
\leq A\int_{x}^{x+h}du\int_{u-h}^{u}e^{\alpha v}dv
={A\over\alpha^2}
\left(
e^{\alpha(x+h)}+e^{\alpha(x-h)}-2e^{\alpha x}
\right).
\end{multline*}
We now apply Lemma 1 in \cite{Mordecki:poli} to the L\'evy process $\alpha X$,
to obtain that condition \eqref{eq:r} implies that
$\E e^{\alpha M}<\infty$.
In consequence, we have
\begin{multline*}
\left|\int_{(h,\infty)}
\left[Q(x_0+h+y)+Q(x_0-h+y)-2Q(x_0+y)\right]dF_M(y)\right|\\
\leq 
{A\over\alpha^2}
\E\left(e^{x_0+h+M}+e^{x_0-h+M}-2e^{x_0+M}\right)\\
\leq 
{A\over\alpha^2}
\left(e^{x_0+h}+e^{x_0-h}-2e^{x_0}\right)\E e^{\alpha M}.
\end{multline*}
In conclussion
$$
\lim_{h\downarrow 0}\int_{(h,\infty)}
\left[Q(x_0+h+y)+Q(x_0-h+y)-2Q(x_0+y)\right]dF_M(y)=0,
$$
concluding the proof of \eqref{eq:smooth}. To verify \eqref{eq:smooth2} the same computations apply with $x$ instead of $x_0$.
The difference is that, in the term \eqref{eq:aa}, we have now
\begin{multline*}
\lim_{h\downarrow 0}\frac1h
[Q(x+h)+Q(x-h)-2Q(x)]\P(M=0)\\=(Q'(x+)-Q'(x-))\P(M=0)=0.
\end{multline*}
This concludes the proof of the Theorem.
\end{proof}
\section{An application}\label{section:example}
To illustrate our results, we consider a compound Poisson process $X=\{X_t\colon t\geq 0\}$ with double-sided exponential jumps, given by
\begin{equation}\label{eq:lp}
X_t=x-\sum_{i=1}^{N^{(1)}_t}Y^{(1)}_i+\sum_{i=1}^{N^{(2)}_t}Y^{(2)}_i,
\end{equation}
where $N^{(1)}=\{N^{(1)}_t\colon t\geq 0\}$ and 
$N^{(2)}=\{N^{(2)}_t\colon t\geq 0\}$ 
are two Poisson processes with respective positive intensities $\lambda_1,\lambda_2$;
$Y^{(1)}=\{Y^{(1)}_i\colon i\geq 1\}$ and $Y^{(2)}=\{Y^{(2)}_i\colon i\geq 1\}$ 
are two sequences of independent exponentially distributed random variables with respective positive parameters 
$\alpha_1,\alpha_2$.
The four processes  $N^{(1)}$, $N^{(2)}$, $Y^{(1)}$, $Y^{(2)}$, are independent.
We consider then the OSP \eqref{eq:osp} for the function  $g(x)=|x|$ and the process $X$.
\subsection{Wiener-Hopf factorization}
The characteristic exponent of $X$ is
$$
\psi(z)=-\lambda_1{z\over\alpha_1+z}+\lambda_2{z\over\alpha_2-z}.
$$
Denote by 
$
-r_1,r_2
$
the roots of the equation $\psi(z)=r$,  that satisfy
$$
-\alpha_1<-r_1<0<r_2<\alpha_2.
$$
We apply the Wiener-Hopf factorization to determine the law of $M$ and $I$
(we do this directly, alternatively we can use the results in \cite{LewisMordecki}).
\begin{equation*}
{r\over r-\psi(z)}
={r_1r_2(\alpha_1+z)(\alpha_2-z)\over \alpha_1\alpha_2(r_1+z)(r_2-z)}.\\
\end{equation*}
In conclusion, due to the uniqueness of the factorization (see Thm. 5(ii) Ch. VI of \cite{Bertoin}),
we obtain
\begin{equation*}
\E e^{zI}={r_1\over\alpha_1}{\alpha_1+z\over r_1+z}={r_1\over\alpha_1}+{\alpha_1-r_1\over\alpha_1}{r_1\over r_1+z},
\end{equation*}
and
\begin{equation*}
\E e^{zM}={r_2\over\alpha_2}{\alpha_2-z\over r_2-z}={r_2\over\alpha_2}+{\alpha_2-r_2\over\alpha_2}{r_2\over r_2-z}.
\end{equation*}
This means that the random variables $M$ and $I$ have defective exponential distributions
with parameters $r_2$ and $-r_1$,
and atoms at zero of respective size $r_2/\alpha_2$ and $r_1/\alpha_1$. 
With a slight abuse of notation,
we denote the respective densities
\begin{align*}
f_I(x)&={r_1\over\alpha_1}\delta_0(x)+{\alpha_1-r_1\over\alpha_1}r_1e^{r_1x},\quad x\leq 0,\\
f_M(x)&={r_2\over\alpha_2}\delta_0(x)+{\alpha_2-r_2\over\alpha_2}r_2e^{-r_2x},\quad x\geq 0.,
\end{align*}
where $\delta_0(x)dx$ denotes the Dirac mass measure at $x=0$.
In order to introduce our result, we also need the following notations.
\begin{align}
E_1&=-\E_0 I=\frac1{r_1}-\frac1\alpha_1>0, 	&E_2	&=\E_0 M=\frac1{r_2}-\frac1\alpha_2>0,\label{eq:e1e2}\\
F_1&=(\E e^{r_2I})^{-1}={\alpha_1\over r_1}{r_1+r_2\over \alpha_1+r_2}>1,&F_2&=(\E e^{-r_1M})^{-1}={\alpha_2\over r_2}{r_1+r_2\over r_1+\alpha_2}>1.\label{eq:f1f2}\\
G_1&=F_1-1={r_2(\alpha_1-r_1)\over r_1(\alpha_1+r_2)}>0,	&G_2&=F_2-1={r_1(\alpha_2-r_2)\over r_2(r_1+\alpha_2)}>0.\label{eq:g1g2}
\end{align}
\begin{theorem}\label{theorem:2}
Consider the L\'evy process $X$ in \eqref{eq:lp}, the payoff function $g(x)=|x|$, and $r>0$.
Denote
\begin{equation}\label{eq:x1}
x_1={E_1(1-e^{-(r_1+r_2)u})+F_1ue^{-r_2u}\over 1+G_1e^{-(r_1+r_2)u}+F_1e^{-r_2u}},
\end{equation}
\begin{equation}\label{eq:x2}
x_2={E_2(1-e^{-(r_1+r_2)u})+F_2ue^{-r_1u}\over 1+G_2e^{-(r_1+r_2)u}+F_2e^{-r_1u}},
\end{equation}
where $u$ is the unique root of the equation
\begin{equation}\label{eq:u2}
u={E_1 +E_2+(E_1G_2+E_2G_1)e^{-(r_1+r_2)u}+E_1F_2e^{-r_1u}+E_2F_1e^{-r_2u}\over1-G_1G_2e^{-(r_1+r_2)u}}.
\end{equation}

Denote
\begin{equation*}
D_1={x_1-x_2e^{-r_2(x_1+x_2)}\over 1-e^{-(r_1+r_2)(x_1+x_2)}},	
\qquad
D_2={x_2-x_1e^{-r_1(x_1+x_2)}\over 1-e^{-(r_1+r_2)(x_1+x_2)}}.	
\end{equation*}
Then, the value function
\begin{equation}\label{eq:ve2}
V(x)=
\begin{cases}
-x,						&\text{ for $x< -x_1$},\\
D_1e^{-r_1(x+x_1)}+D_2e^{r_2(x-x_2)},	&\text{ for $-x_1\leq x\leq x_2$},\\
x,						&\text{ for $x_2<x$},
\end{cases}
\end{equation}
and stopping time defined in \eqref{eq:tau0}  conform the solution of the OSP \eqref{eq:osp}.
\end{theorem}
\begin{remark}An application of Theorem \ref{theorem:angle} (or more directly in this case, 
the computation of the corresponding derivatives in formula \eqref{eq:ve2})
show that the smooth pasting condition does not hold in any of the thresholds of the problem: the averaging functions
have non-vanishing derivatives at the roots, and both the maximum and the infimum have atoms at the origin.
\end{remark}
\begin{proof}
We first verify that in fact equation \eqref{eq:u2} has only one positive root, as the r.h.s. decreases from
$$
{2(E_1F_2+E_2F_1)\over1-G_1G_2}>E_1+E_2>0,
$$
when $u=0$ to
$
E_1 +E_2,
$
as $u\to\infty$. Here it was used that $G_1G_2<1$, fact that follows directly from the definitions \eqref{eq:g1g2}.
Furthermore, after some computations, it can be checked that 
$u$ is a root of \eqref{eq:u2} if and only if it is a root of the equation
$$
u={E_1(1-e^{-(r_1+r_2)u})+F_1ue^{-r_2u}\over 1+G_1e^{-(r_1+r_2)u}+F_1e^{-r_2u}}+{E_2(1-e^{-(r_1+r_2)u})+F_2ue^{-r_1u}\over 1+G_2e^{-(r_1+r_2)u}+F_2e^{-r_1u}},
$$
so, according to definitions  \eqref{eq:x1} and \eqref{eq:x2}, we obtain that $u=x_1+x_2$. 
In order to apply Theorem \ref{theorem:1}, introduce the functions
\begin{align}\notag
Q_1(x)&=
\begin{cases}
-x-E_1-F_1D_2e^{r_2(x-x_2)},	&\text{ for $x\leq -x_1$},\\
0,					&\text{ for $x> -x_1$},
\end{cases}\\
Q_2(x)&=
\begin{cases}
0,					&\text{ for $x\leq x_2$},\\
x-E_2-F_2D_1e^{-r_1(x+x_1)},	&\text{ for $x\geq x_2$},\label{eq:q2}
\end{cases}
\end{align}
We prove that these two functions are monotonous, non-increasing and non-decreasing respectively, and continuous.
It can be checked that 
\begin{equation*}
x_1={E_1(1-e^{-(r_1+r_2)u})+F_1ue^{-r_2u}\over 1+G_1e^{-(r_1+r_2)u}+F_1e^{-r_2u}},
\end{equation*}
if and only if
\begin{equation}\label{eq:iff}
x_1=E_1+F_1{x_2-x_1e^{-r_1u}\over 1-e^{-{(r_1+r_2)}u}}e^{-r_2u}=E_1+F_1D_2e^{-r_2(x_1+x_2)},
\end{equation}
and this last statement \eqref{eq:iff} is equivalent to $Q_1(-x_1)=0$, giving the continuity of $Q_1$, that is clearly positive and non-increasing.
Identical arguments apply to $Q_2$.

To check equality \eqref{eq:equal}, consider first $x\geq -x_1$ and compute
\begin{multline*}
\E_xQ_1(I)=\int_{(-\infty,0]}Q_1(x+y)f_I(y)dy\\
=\int_{-\infty}^{x}Q_1(z)f_I(z-x)dy
=\int_{-\infty}^{-x_1}Q_1(z)f_I(z-x)dy\\
={\alpha_1-r_1\over\alpha_1}e^{-r_1(x+x_1)}\left[
x_1+{1\over\alpha_1}
-{\alpha_1\over \alpha_1+r_2}D_2e^{-r_2(x_1+x_2)}
\right],
\end{multline*}
where we used the definition of $F_1$ in \eqref{eq:f1f2}. Observing that $D_1$ and $D_2$ are the solutions
of the linear system of equations
\begin{equation}\label{eq:system}
\begin{cases}
D_1+D_2e^{-r_2(x_1+x_2)}=x_1,\\
D_1e^{-r_1(x_1+x_2)}+D_2=x_2,
\end{cases}
\end{equation}
we obtain
\begin{multline}\label{eq:d0}
\E_xQ_1(I)={\alpha_1-r_1\over\alpha_1}e^{-r_1(x+x_1)}
\left[
x_1+{1\over\alpha_1}
-{\alpha_1\over\alpha_1+r_2}(x_1-D_1)
\right]
\\
={\alpha_1-r_1\over\alpha_1}e^{-r_1(x+x_1)}
\left[
{r_2\over\alpha_1+r_2}x_1+{1\over\alpha_1}
-{\alpha_1\over\alpha_1+r_2}D_1
\right].
\end{multline}
%
%
Regarding the maximum, now for $x>x_2$, 
\begin{multline*}
\E_xQ_2(M)={\alpha_2\over r_2}Q_2(x)+e^{r_2x}\int_x^{\infty}Q_2(z)f_I(z)dz\\
={r_2\over\alpha_2}\left[
x-E_2-F_2D_1e^{-r_1(x+x_1)}
\right]
+{\alpha_2-r_2\over\alpha_2}
\left[
x+\frac1{\alpha_2}-{\alpha_2\over r_1+\alpha_2}D_1e^{-r_1(x+x_1)}
\right]\\
=x-{r_2\over\alpha_2}F_2D_1e^{-r_1(x+x_1)}
-{\alpha_2-r_2\over r_1+\alpha_2}D_1e^{-r_1(x+x_1)}
\\
=x-D_1e^{-r_1(x+x_1)}.
\end{multline*}
Summing up, 
\begin{multline*}
\E_xQ_1(I)+\E_xQ_2(M)\\=x+e^{-r_1(x+x_1)}
\left\{{\alpha_1-r_1\over\alpha_1}
\left[
{r_2\over\alpha_1+r_2}x_1+{1\over\alpha_1}
-{\alpha_1\over\alpha_1+r_2}D_1
\right]-D_1\right\}.
\end{multline*}
To show that the second summand in the r.h.s. of the previous formula vanishes, 
we compute
\begin{multline*}
{\alpha_1-r_1\over\alpha_1}
\left[
{r_2\over\alpha_1+r_2}x_1+{1\over\alpha_1}
-{\alpha_1\over\alpha_1+r_2}D_1
\right]-D_1\\
=
x_1{\alpha_1-r_1\over\alpha_1}{r_2\over\alpha_1+r_2}
+{\alpha_1-r_1\over\alpha_1^2}
+{\alpha_1-r_1\over\alpha_1+r_2}D_1-D_1\\
=
x_1{r_2\over\alpha_1}{\alpha_1-r_1\over\alpha_1+r_2}
+{\alpha_1-r_1\over\alpha_1^2}
-{r_1+r_2\over\alpha_1+r_2}D_1.
\end{multline*}
Taking into account \eqref{eq:iff} and the first equation in \eqref{eq:system}, we obtain
$$
x_1={D_1F_1-E_1\over F_1-1}={r_1\over r_2}{\alpha_1+r_2\over\alpha_1-r_1}(D_1F_1-E_1),
$$
and substitute
\begin{multline*}
x_1{r_2\over\alpha_1}{\alpha_1-r_1\over\alpha_1+r_2}
+{\alpha_1-r_1\over\alpha_1^2}
-{r_1+r_2\over\alpha_1+r_2}D_1
\\
={r_1\over\alpha_1}(D_1F_1-E_1)
+{\alpha_1-r_1\over\alpha_1^2}
-{r_1+r_2\over\alpha_1+r_2}D_1
\\
=D_1\left({r_1\over\alpha_1}F_1-{r_1+r_2\over\alpha_1+r_2}\right)
-{r_1\over\alpha_1}E_1
+{\alpha_1-r_1\over\alpha_1^2}=0,
\end{multline*}
in view of the definitions of $E_1$ in \eqref{eq:e1e2} and $F_1$ in \eqref{eq:f1f2}.
Similarly, for $x<-x_1$, we obtain $\E_xQ_1(I)+Q_2(M)=-x$. This concludes the verification of \eqref{eq:equal}.

We now verify \eqref{eq:geq}. From the computations above, we obtained that
\begin{equation}\label{eq:d1}
D_1={\alpha_1-r_1\over\alpha_1}
\left[
{r_2\over\alpha_1+r_2}x_1+{1\over\alpha_1}
-{\alpha_1\over\alpha_1+r_2}D_1
\right].
\end{equation}
This gives that, for $x\geq -x_1$, we have (for further reference we also write the formula for $M$)
\begin{equation}\label{eq:d1d2}
\E_xQ_1(I)=D_1e^{-r_1(x+x_1)},\qquad
\E_xQ_2(M)=D_2e^{r_2(x-x_2)}.
\end{equation}
In particular, this gives that $D_1>0$, and, for $-x_1\leq x\leq x_2$, 
$$
\E_xQ_1(I)+\E_xQ_2(M)=D_1e^{-r_1(x+x_1)}+D_2e^{r_2(x-x_2)},
$$
so the definition \eqref{eq:ve} gives \eqref{eq:ve2}. To conclude with the proof, it remains to verify
condition \eqref{eq:geq} for $x\in[-x_1,x_2]$. 
We take $0\leq x\leq x_2$. We have
\begin{align*}
\E_xQ_2(M)	&=\E_0\left(x+M-E_2-F_2D_1e^{-r_1(x+M+x_1)}\right)^+\\
		&\geq \E_0\left(x+M-E_2-F_2D_1e^{-r_1(M+x_1)}\right)\\
		&=x+\E_0M-E_2-D_1F_2\E_0\left(e^{-r_1M}\right)e^{-r_1(x+x_1)}\\
		&=x-D_1e^{-r_1(x+x_1)}=x-\E_xQ_1(I),
\end{align*}
in view of the definition of $F_2$ in \eqref{eq:f1f2}, and \eqref{eq:d0} and\eqref{eq:d1}.
For $-x_1\leq x\leq 0$ the symmetric computation completes the verification of \eqref{eq:geq}.
This completes the verification of all the hypothesis of Theorem \ref{theorem:1}, and the proof of Theorem \ref{theorem:2}.
\end{proof}
\begin{remark}
It is interesting to note that the function $x-E_2$ appearing in the first two summands in the r.h.s. of \eqref{eq:q2},
in the terminology of \cite{Surya},
is the averaging function  of the one sided problem with payoff
function $g_2(x)=x^+$ (see \cite{Mordecki:2002}). So the remaining term 
in the r.h.s. in \eqref{eq:q2} is a correction due to the presence of the infimum in the two-sided problem.
As $x_2$ is the root of $Q_2$, we obtain $x_2\geq E_2$.
\end{remark}

\begin{remark} The proof of Theorem \ref{theorem:2} also provides bounds to find $u$ numerically. 
First observe that $Q_2(x)\leq x$ for $x\geq 0$, so $\E_0Q_2(M)\leq E_2$.
Furthermore, in view of \eqref{eq:iff} and \eqref{eq:d1d2}, we have
$$
x_1=E_1+F_1D_2e^{-r_2(x_1+x_2)}=E_1+F_1\E_0Q_2(M)e^{-r_2x_2}\leq E_1+F_1E_2.
$$ 
In view of the previous remark, the conclusion is that
$$
E_1+E_2\leq u\leq E_1(1+F_2)+E_2(1+F_1).
$$
\end{remark}
\subsection{Numerical examples}
To illustrate our results we consider two examples. In the first one we choose 
$
(\alpha_1,\lambda_1,\alpha_2,\lambda_2,r)=(1,3,3,1,1).
$
The thresholds are
$x_1=1.17$ and $x_2=0.87$.
The second example is symmetric, with 
$
(\alpha_1,\lambda_1,\alpha_2,\lambda_2,r)=(1,1,1,1,1)
$
The critical thresholds are $x_1=x_2=1.04$.
The corresponding value functions \eqref{eq:ve2} are shown in Figure \ref{figure:1}.
\begin{figure}
\centering
\includegraphics[scale=0.2]{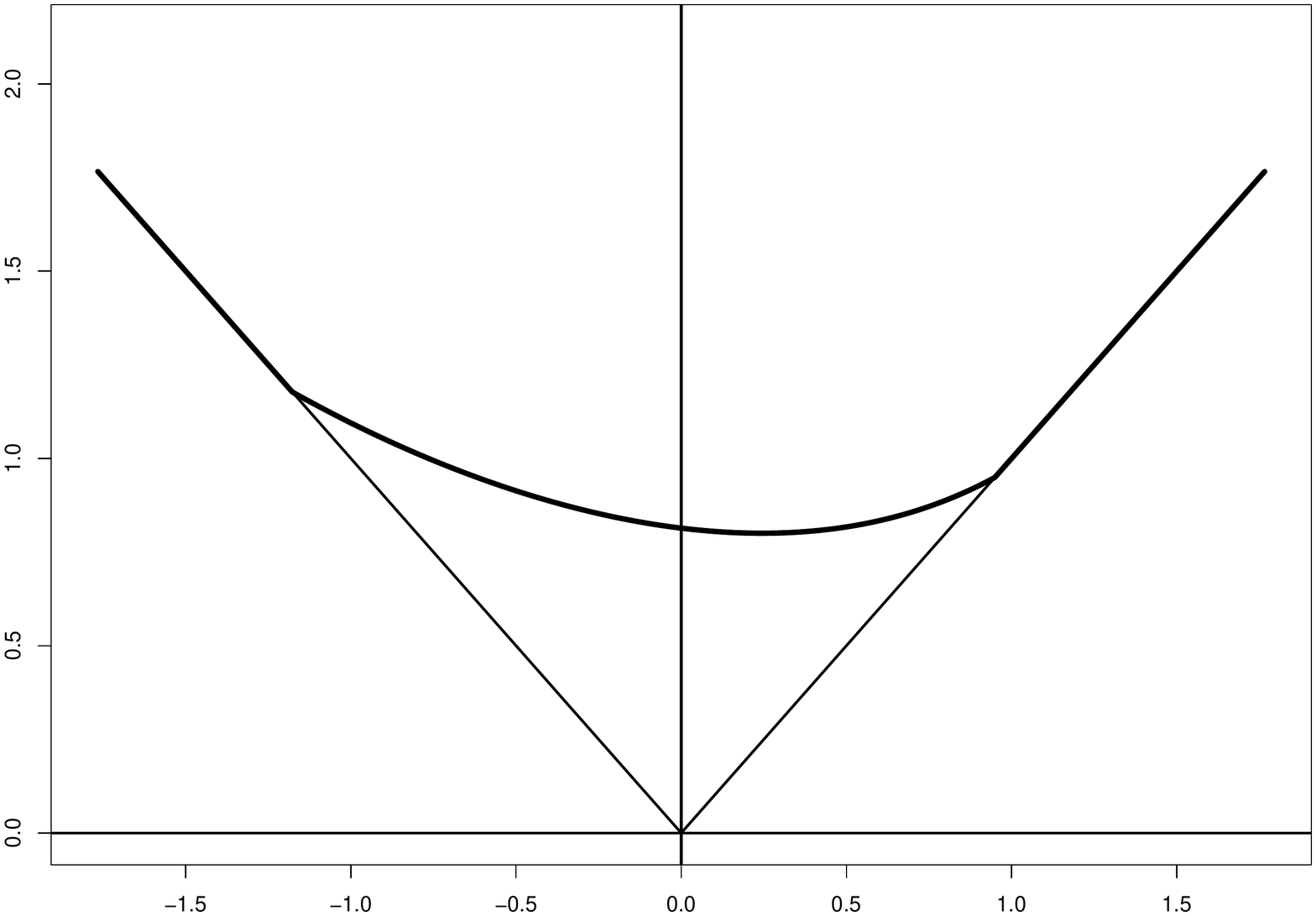}
\includegraphics[scale=0.2]{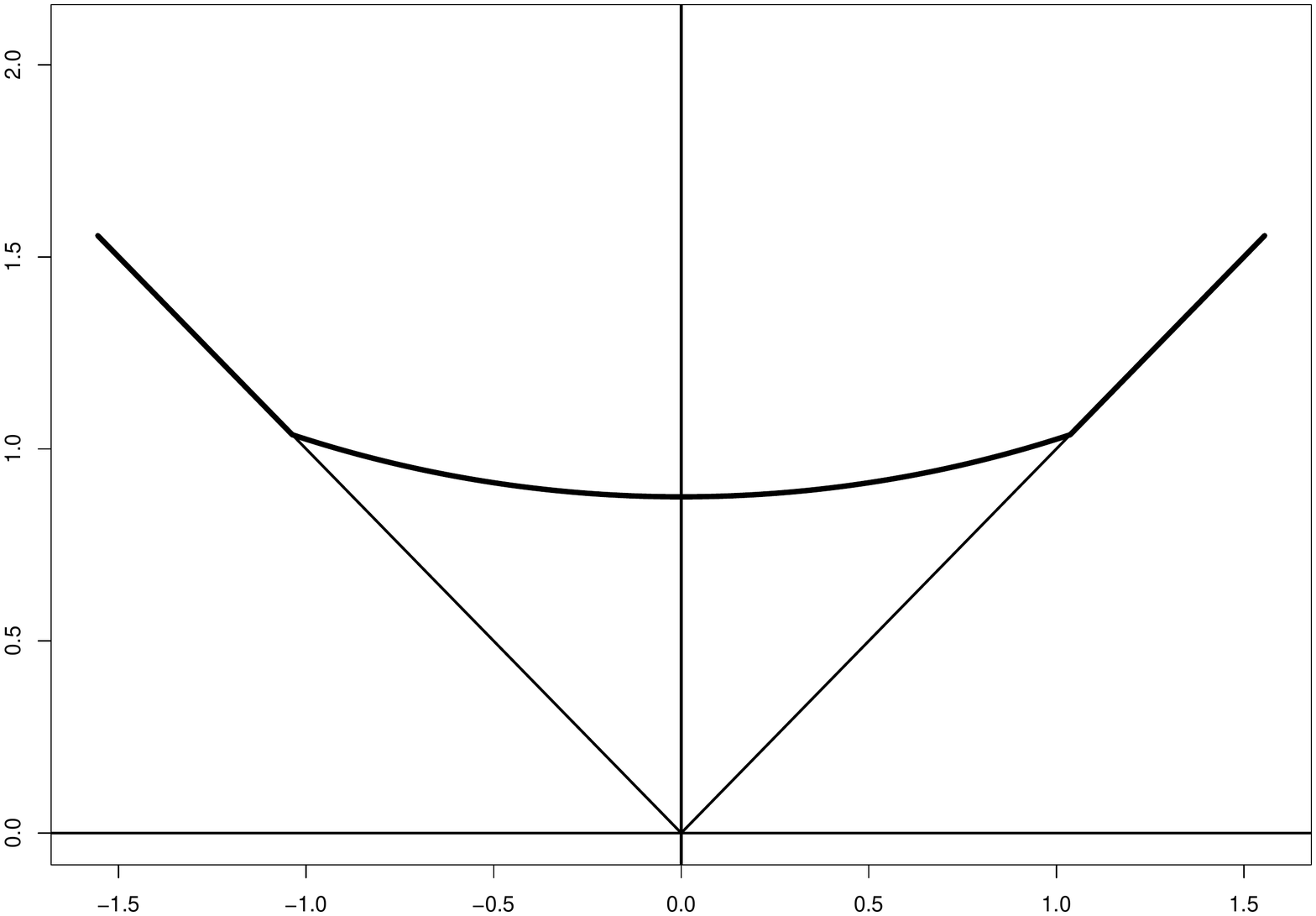}
\caption{The payoff functions $V$ for the set of parameters $(1,3,3,1,1)$ (left) and $(1,1,1,1,1)$ (right).}\label{figure:1}
\end{figure}

\end{document}